\documentclass[11pt,english]{article}
\usepackage{lmodern}

\usepackage[T1]{fontenc}
\usepackage[utf8]{inputenc}
\usepackage[a4paper]{geometry}
\usepackage{color}
\definecolor{document_fontcolor}{rgb}{0, 0, 0}
\color{document_fontcolor}
\usepackage{amsmath}
\usepackage{amsthm}
\usepackage{amssymb}
\usepackage{stmaryrd}
\usepackage{stackrel}
\usepackage{setspace}
\usepackage{esint}
\makeatletter
\numberwithin{equation}{section}
\theoremstyle{plain}
\newtheorem{thm}{\protect\theoremname}[section]
\theoremstyle{plain}
\newtheorem{lem}[thm]{\protect\lemmaname}
\ifx\proof\undefined
\newenvironment{proof}[1][\protect\proofname]{\par
\normalfont\topsep6\p@\@plus6\p@\relax
\trivlist
\itemindent\parindent
\item[\hskip\labelsep\scshape #1]\ignorespaces
}{%
\endtrivlist\@endpefalse
}
\providecommand{\proofname}{Proof}
\fi
\theoremstyle{plain}
\newtheorem{prop}[thm]{\protect\propositionname}

\makeatother

\usepackage{babel}
\providecommand{\lemmaname}{Lemma}
\providecommand{\propositionname}{Proposition}
\providecommand{\theoremname}{Theorem}

\begin{document}

\title{Kolmogorov-Arnold-Moser Theorem}

\author{Othmane Islah}

\maketitle
This article is about the proof of the celebrated KAM theorem as
sketched out in \cite{KOL} Kolmogorov's original presentation to the ICM. The proof\footnote{Cedric Villani brought to my attention that Chierchia has also done this in his article  } presented here  has been detailed as an effort to clarify if  Kolmogorov's argument can be made rigorous. This is a legitimate question because the published proof \cite{ARN} a few years later by his then student, the eminent mathematician V.I Arnold, did not follow the same plan

\section{Notations and preliminaries}
\begin{itemize}
\item Let U be an open set of $\mathbb{R^{\mathrm{d}}}$ containing the
origin and $\mathbf{\mathbb{T^{\mathrm{d}}}}$ the d-dimensional torus.
\item $|.|$ is the sup norm in $\mathbb{R^{\mathrm{d}}}$ , $\mathbb{C^{\mathrm{d}}}$
, $\mathcal{\mathcal{M_{\mathrm{d}}}\left(\mathbb{C}\right)}$ ~and
in spaces of analytical functions defined on a compact set. 
\item $\mathcal{H_{\mathrm{\rho,\Delta}}}$ is the set of analytical functions
defined on $\mathrm{U}\times\mathbb{\mathbb{T^{\mathrm{d}}}}$~that
have a holomorphic extension to $\mathrm{A_{\rho,\Delta}}$ : 
\[
\mathrm{A_{\rho,\Delta}}=\left\{ (x,y)\in\mathbb{C}^{\mathrm{2d}}/\,|x|\leq\rho\textrm{,|}\mathrm{Im}(y)|\leq\Delta\right\} 
\]

\item $\mathcal{H_{\mathrm{\Delta}}}$ is the set of analytical functions
of $\mathbb{\mathbb{T^{\mathrm{d}}}}$~that have a holomorphic extension
to $\mathrm{A_{\Delta}}=\left\{ x\in\mathbb{C}^{\mathrm{d}}/,\,|\mathrm{Im}(x)|\leq\Delta\right\} $
\item $|.|_{\rho,\Delta}$ and $|.|_{\Delta}$are the sup norms in $\mathcal{H_{\mathrm{\rho,\Delta}}}$,
$\mathcal{H_{\mathrm{\Delta}}}$ on the compacts $\mathrm{A_{\rho,\Delta}}$,
$\mathrm{A_{\Delta}}$.
\item $\mathcal{H_{\mathrm{\Delta}}^{\mathrm{0}}}$ is the subset of elements~$f$
~of $\mathcal{H_{\mathrm{\Delta}}}$ such that ${\displaystyle \underset{\mathrm{\mathbb{T}^{d}}}{\int}f(\theta)\mathrm{d}\theta}=0$.
\item ${\displaystyle DC(c,\tau)=\left\{ \omega\in\mathbb{R}^{d}\,\,/\forall k\in\mathbb{Z^{\mathrm{d}}}\backslash\{0\},\,\mid\omega.k\mid\geq\frac{c}{|k|^{\tau}}\right\} }\,$
is the set of diophantine vectors with parameters $c$ and $\tau>d-1$. 
\end{itemize}

\section{The Classical Kolmogorov-Arnold-Moser Theorem}

For $\epsilon>0$ and $f_{0},\,f_{1}\in\mathcal{H_{\mathrm{\rho,\Delta}}}$,
let $H_{\epsilon}(r,\theta)=f_{0}(r)+\epsilon f_{1}(r,\theta)$ ~be
a perturbation of the integrable system given by $H_{0}$ for which
the tori $\{r_{0}\}\times\mathbf{\mathbb{T^{\mathrm{d}}}}$, $r_{0}\in\mathbb{R}{}^{d}$
are invariant.

The aim of KAM theory is to establish the survival of most of the
invariant tori after perturbation, with of course a small deformation
compared to the flat invariant tori of $H_{0}$.

The theory assert that under various non degeneracy or 'twist conditions'
on $f_{0}$ most of the tori survive after perturbation, in the sense
that the relative measure inside $B_{\rho}\times\mathbb{T}^{d}$ of
the invariant tori of $H_{\epsilon}$ tends to $1$ as $\epsilon$
tends to $0$.

We will present the oldest result in that direction, namely Kolmogorov's
theorem that a Diophantine invariant torus of a completely integrable
analytic Hamiltonian with a non-degenrate Hessian survives under analytic
perturbations.

To fix notations, we will be interested in the invariant torus $\{0\}\times\mathbf{\mathbb{T^{\mathrm{d}}}}$
that has frequency 
\[
{\displaystyle \omega:=\frac{\partial f_{0}}{\partial r}(0)}.
\]

We will always assume that $\omega\in{\rm DC}(c,\tau)$.

We let 
\[
{\displaystyle \hat{S}:=\frac{\partial^{2}f_{0}}{\partial r^{2}}(0)}.
\]

We say that $f_{0}$ has the Kolmogorov twist condition on the torus
$\{0\}\times\mathbf{\mathbb{T^{\mathrm{d}}}}$ if $\hat{S}$ is an
invertible matrix. In the sequel, we will always assume that this
is the case.

The following theorem asserts that under these two conditions, for
any small $\epsilon>0,$ $H_{\epsilon}$ has an invariant torus of
frequency $\omega$ that is a small perturbation of the invariant
torus $\{0\}\times\mathbf{\mathbb{T^{\mathrm{d}}}}\,$ of the original
system. We say that the torus with frequency $\omega$ is persistent
after small perturbation.
\begin{thm}
\label{thm:Kolmogorov-Arnold-Moser}\textbf{\textup{Kolmogorov}}
\end{thm}
\noindent There exists $\epsilon_{0}(\rho,\Delta,c,\tau,\hat{S},|f_{0}|_{\Delta},|f_{1}|_{\rho,\Delta})$
such that if $\epsilon<\epsilon_{0}$ then there exists a real analytic
symplectic diffeomorphism $\phi_{\epsilon}:A_{\rho/3,\Delta/3}\rightarrow{\displaystyle \mathrm{V}\subset\mathrm{A_{\rho,\Delta}}}$
such that

\begin{equation}
H_{\epsilon}\circ\phi_{\epsilon}(r,\theta)=\omega.r+g(r,\theta)\label{eq:equation change of variable}
\end{equation}
with $g=O\left(r^{2}\right)$.

\noindent ~Also : 
\[
{\displaystyle \mid\phi_{\epsilon}-\mathrm{Id}\mid}_{\frac{\rho}{3},\frac{\Delta}{3}}\leq C\epsilon
\]
with $C:=C(\rho,\Delta,c,\tau,\hat{S},|f_{0}|_{\Delta},|f_{1}|_{\rho,\Delta})$.

\paragraph{Remarks}
\begin{enumerate}
\item The existence of a symplectic diffeormorphism close to the identity,
means that the invariant torus of the perturbed system is at a $O(\epsilon)$
distance to the invariant torus of the initial integrable system.
Moreover, the quasi-periodic orbits of the perturbed system remain
close to the quasi-periodic orbits of the integrable system at every
time.
\item We could have proven an apparently stronger result without additionnal
effort where the domain of $\phi_{\epsilon}$ could be set to any
$A_{\tilde{\rho},\tilde{\Delta}}$ with $0<\tilde{\rho}<\rho$ and
$0<\tilde{\Delta}<\Delta$.
\item The statement of the theorem remains of course valid for any invariant
torus of the integrable system as long as the two conditions - $\frac{\partial h_{0}}{\partial r}(r_{0})$
diophantian and $\frac{\partial^{2}h_{0}}{\partial r^{2}}(r_{0})$
invertible - are satisfied.
\item It is a known fact that the set of diophantian frequencies has full
Lebesgue-measure in any ball in $\mathbb{R}^{d}$.
\end{enumerate}
The idea of the proof is to construct iteratively the change of coordinates
by solving a linearised form of the equation (\ref{eq:equation change of variable}).
At each step, we will cancel successively the contribution of the
angular variable to the order $\epsilon^{2}$, $\epsilon^{4}$,...,$\epsilon^{2^{n}}$
and at the same time ensure that the conditions (unchanged diophantian
frequency, invertible Hessian) for the first and second order term
in $r$ remain valid. This is a quadratic scheme for which the fast
convergence offsets the loss of analyticity due to the small divisors
when we solve the linearised form of the equation.\\
The linearised equation will lead to the following linear partial
differential equation :

\[
\mathcal{L_{\omega}}f=g\,\,\textrm{with}\,\,g\in\mathcal{H_{\mathrm{\Delta}}^{\mathrm{0}}},\,\,f\in\mathcal{H_{\mathrm{\Delta-\delta}}^{\mathrm{0}}}\,,\,0<\delta<\Delta
\]
Where: 
\[
\mathcal{L_{\omega}}=\stackrel[i=1]{d}{\sum}\omega_{i}\frac{\partial}{\partial\theta_{i}}
\]
 . \\
The following fundamental lemma shows that this equation has a unique
solution if $\omega$ is diophantian and provides an estimate on the
sup norm of the solution.
\begin{lem}
\label{lem:Fundamental-lemma}\textbf{\textup{Fundamental lemma}}

If~${\displaystyle \omega\in DC(c,\tau)}$ then: 
\[
{\displaystyle \forall g\in\mathcal{H_{\mathrm{\Delta}}^{\mathrm{0}}}\,\,\forall\delta>0,\,\,\delta<\Delta,\,\,\exists!\,f\in\mathcal{H_{\mathrm{\Delta-\delta}}^{\mathrm{0}}\,\,}}
\]
 such that 
\begin{equation}
{\displaystyle \mathcal{L_{\omega}}f=g}\label{eq:cohomol}
\end{equation}
. 

Moreover ${\displaystyle \exists\,C(\tau,\,d)>0}$ such that : 
\begin{equation}
\mid f\,\mid_{\Delta-\delta}\leq\frac{C(\tau,\,d)}{\delta^{d+\tau}}\,\mid g\,\mid_{\Delta}\label{eq:estimate sol cohom}
\end{equation}
 .\end{lem}
\begin{proof}
If we take the Fourier expansions of $f$ and $g$ and substitute
in the equation, we obtain :
\[
\forall\,k\in\mathbb{Z}^{d}\backslash\{0\}\,\,(\omega.k).f_{k}=g_{k}
\]
The unicity follows from this.\\
 Given the analyticity of $g$, we have that :

\begin{equation}
\forall\,k\in\mathbb{Z}^{d}\backslash\{0\}\,\,\mid g_{k}\mid\leq\mid g\mid_{\triangle}\exp(-\Delta.\mid k\mid)\label{eq:1}
\end{equation}
By setting $\forall z=x+iy\in\mathrm{A_{\Delta-\delta}}$ 
\[
f\,(x+iy)=\sum_{k\in\mathbb{Z}^{d}\backslash\{0\}}\,\frac{g_{k}}{\omega.k}\,e^{ik.z}
\]
We see that the series is convergent because get from (\ref{eq:1})
and the diophantine condition:
\[
\forall k\in\mathbb{Z}^{d}\backslash\{0\}\,\,\forall z\in\mathrm{A_{\Delta-\delta}}\,\,\,\biggl|\frac{g_{k}}{\omega.k}\,e^{ik.z}\biggr|\leq1/c.\bigl|k\bigr|^{\tau}\exp(-\delta\bigl|k\bigr|)\mid g\mid_{\triangle}
\]
Therefore $f\in\mathcal{H_{\mathrm{\Delta-\delta}}^{\mathrm{0}}}$
and we have :
\[
\mid f\,\mid_{\Delta-\delta}\leq\mid g\mid_{\triangle}\sum_{k\in\mathbb{Z}^{d}\backslash\{0\}}1/c.\bigl|k\bigr|^{\tau}\exp(-\delta\bigl|k\bigr|)\leq\frac{\mid g\mid_{\triangle}}{c\,\delta^{\tau+d}}\intop_{\mathbb{R}^{d}}|x|^{\tau}e^{-|x|}dx
\]

\end{proof}

\section{Analysis of the linearised equation~}

We will show in this section how to solve the linearised problem by
proving proposition \ref{prop:Linearised} but we will carry out some
ad hoc calculations to introduce the end result in a more natural
way.

As indicated in Kolmogorov's article, for the construction of the
analytical symplectic diffeomorphism, we will introduce the following
generating function:

\[
A(\theta,\,R)=\mathrm{u}(\theta)+\alpha.\theta+^{t}(\theta+\mathrm{v(\theta)).}R
\]

where $\mathrm{u}:\,\mathrm{\mathbb{T}^{d}\rightarrow\mathbb{R}}$
and $\mathrm{v}:\,\mathrm{\mathbb{T}^{d}\rightarrow\mathbb{R}^{d}}$
are analytical, $R\in\mathbb{R}^{d}$, and $\alpha\in\mathbb{R}^{d}$
is a constant.

We are seeking a canonical transformation ${\displaystyle \phi:\,(r,\,\theta)\rightarrow(R\,,\,\varphi)}$
with a domain to be determined such that ${\displaystyle \frac{\partial A}{\partial R}=\varphi}$
and ${\displaystyle \frac{\partial A}{\partial\theta}=r}$ so :

\begin{align}
\varphi= & \theta+\mathrm{v(}\theta)\label{eq:symplectic map}\\
r= & \frac{\partial\mathrm{u}}{\partial\theta}+\alpha+R+^{t}d\mathrm{v}.R\label{eq:symplectic map2}
\end{align}

Given that we are looking for a diffeomorphism close to the identity
at the order $\epsilon$, $\mathrm{u}$,$\mathrm{v}$ and $\alpha$
will be of the order $\epsilon$ as will be seen below.

We will write the perturbed hamiltonian $H_{\epsilon}$ in a different
form to distinguish the linear and quadratic part in $r$ from the
terms of order $\epsilon$ and the terms of order higher than 3 in
$r$:

\begin{equation}
H_{\epsilon}(r\mathrm{,}\theta)=m+{}^{t}\omega.r\,+\frac{1}{2}\,^{t}r.S(\theta).r+\epsilon\,h(\theta,r)+g(\theta,r)\label{eq:ham}
\end{equation}

Where $m$ is a constant, $\omega$ is diophantian, $S(\theta)$ a
symmetric matrix, $h,\,g\,\in\mathcal{H_{\mathrm{\rho,\Delta}}}$
and $g(\theta,r)=O\left(r^{3}\right)$.

Additionnally, we will assume that we are given an invertible symmetric
matrix $\hat{S}$ such that :
\begin{equation}
\begin{cases}
\exists\,\eta>0 & B(\hat{S},\,\eta)\subset Gl_{n}(\mathbb{C})\\
\forall\,\theta\in\mathbb{T}{}^{d},\, & S(\theta)\in B(\hat{S},\,\eta)
\end{cases}\label{eq:non degen neighboor}
\end{equation}
With~$B(\hat{S},\,\eta)$, the ball of center $\hat{S}$ abd radius
$\eta$.

Also, we set $\Delta R=\frac{\partial\mathrm{u}}{\partial\theta}+\alpha+^{t}d\mathrm{v}.R$
and $R_{t}=R+t\,\Delta R$ for $0\leq t\leq1$\@. We will assume
in the remainder of this section that $\Delta R=O\left(\epsilon\right)$
.

First we notice that :

\begin{equation}
\epsilon\,h(R+\Delta R,\,\theta)=\epsilon\,h(R,\,\theta)+\epsilon\,\intop_{0}^{\,\,1}t\,\,\frac{\partial h}{\partial R}(R_{t}\,,\theta)\,\mathrm{dt}\,\Delta R\label{eq: epsilon term}
\end{equation}

Since $\Delta R=O\left(\epsilon\right)$ , the second term on the
rhs is $O\left(\epsilon^{2}\right)$.

Let's expand $h$ as:

\[
h(R,\,\theta)=a(\theta)+^{t}b(\theta).R+\frac{1}{2}\,^{t}R.c(\theta).R+\sum_{|k|\geq3}\,h_{k}(\theta)\,R^{k}
\]

By Taylor formula applied to $g$ we have :

\begin{equation}
g\,(R+\Delta R,\,\theta)=g(R,\,\theta)\,+^{t}\Delta R\,.\frac{\partial g}{\partial R}(R,\theta)+\frac{1}{2}\,^{t}\Delta R\left(\intop_{0}^{\,\,1}\,t^{2}\,\frac{\partial^{2}g}{\partial R^{2}}(R_{t}\,,\theta)\,\mathrm{dt}\right)\Delta R\label{eq:order 3 term}
\end{equation}

On the rhs,the first term is $O\left(R^{3}\right)$, the second one
is $O\left(R^{2}\right)$ and the last term is $O\left(\epsilon^{2}\right)$
because $\Delta R=O\left(\epsilon\right)$ .

Also if we write $g(R,\theta)=\sum_{|k|\geq3}\,g_{k}(\theta)\,R^{k}$,
we have that :

\begin{equation}
^{t}\Delta R\,.\frac{\partial g}{\partial R}(R,\theta)=\sum_{k=1}^{d}\,^{t}R.Q_{k}(\theta).R.(\Delta\tilde{R}+^{t}\mathrm{dv}.R)_{k}+^{t}\Delta R.\frac{\partial}{\partial R}\left(\sum_{|k|\geq4}\,g_{k}(\theta)\,R^{k}\right)\label{eq:quad form order 3 term}
\end{equation}

where $\forall\,i,\,j,k$: 
\begin{eqnarray*}
Q_{k}(\theta) & = & g_{kij}(\theta)=\frac{1}{6}\frac{\partial g}{\partial R_{k}\partial R_{i}\partial R_{j}}(0,\,\theta)\\
\Delta R & = & \Delta\tilde{R}+^{t}\mathrm{dv}.R
\end{eqnarray*}

Finally for the quadratic form $S$ , we have : 
\begin{equation}
\frac{1}{2}\,^{t}(R+\Delta R).S(\theta).(R+\Delta R)=\frac{1}{2}\,^{t}R.S(\theta).R\,+\,^{t}R.S(\theta).\Delta R+\,\frac{1}{2}\,^{t}\Delta R.S(\theta).\Delta R\label{eq: quad form}
\end{equation}

If we set $\bar{H}(R,\theta)=H_{\epsilon}(R+\Delta R,\theta)$, we
then collect the terms up to the first order in $R$ and $\epsilon$
to obtain :

\begin{eqnarray*}
\bar{H}(R,\theta) & = & m+^{t}\omega.\alpha+^{t}\omega.R+\mathcal{L_{\omega}}u(\theta)+\epsilon a(\theta)+^{t}R.\left(\mathrm{dv}(\theta).\omega+S(\theta).\left(\frac{\partial u}{\partial\theta}+\alpha\right)+\epsilon\,b(\theta)\right)+\\
 &  & O\left(\epsilon^{2}\right)+O\left(R^{2}\right)
\end{eqnarray*}

In order to cancel up to the first order the dependence on $\theta$,
we must solve the following equations :

\begin{equation}
\begin{cases}
\mathcal{L_{\omega}}u= & -\epsilon a(\theta)+\epsilon\intop_{\mathrm{T^{d}}}a(\theta)\mathrm{d}\theta\\
\mathcal{L_{\omega}}\mathrm{v}= & -S(\theta).\left(\frac{\partial u}{\partial\theta}+\alpha\right)-\epsilon\,b(\theta)\\
\left(\intop_{\mathrm{T^{d}}}S(\theta)\mathrm{d}\theta\right)\alpha & =-\intop_{\mathrm{T^{d}}}\left(S(\theta)\frac{\partial u}{\partial\theta}+\epsilon\,b(\theta)\right)\mathrm{d}\theta
\end{cases}\label{eq:system of fundamental equations}
\end{equation}

where $\mathcal{L_{\omega}}\mathrm{v}=^{t}(\mathcal{L_{\omega}}\mathrm{v}_{1},...,\mathcal{L_{\omega}}\mathrm{v}_{d})$
.

The third equation is to ensure that the rhs of the second equation
has a zero average, and has a unique solution because $\intop_{\mathrm{T^{d}}}S(\theta)\mathrm{d}\theta\in B(\hat{S},\,\eta)$,
thus invertible. The first two equations have then a unique solution
by the fundamental lemma (\ref{lem:Fundamental-lemma}) and can be
defined on $\mathrm{A_{\Delta-\delta}}$ for any $0<\delta<\Delta$
. Also, the map defined by equation (\ref{eq:symplectic map}) is
injective on $\mathrm{A_{\Delta-\delta}}$ , if ~$\mid\mathrm{dv}\mid_{\Delta-\delta}<1$
and it defines then a diffeomorphism on its image. Under this condition,
we obtain a symplectic diffeomorphism $(R,\varphi)=\phi(r,\theta)$
from $\mathrm{A_{\rho,\Delta-\delta}}$ onto its image given by :
\begin{equation}
\begin{cases}
\varphi= & \theta+\mathrm{v}(\theta)\\
R\,= & \left(\,I+^{t}\mathrm{dv}(\theta)\,\right)^{-1}.\left(\,r-\alpha-\frac{\partial u}{\partial\theta}\right)
\end{cases}\label{eq:symplectic diffeomorphism}
\end{equation}

Where $u$, $\mathrm{v}$ and $\alpha$ are solutions of the system
of equations(\ref{eq:system of fundamental equations}).

Then for $(R,\varphi)\in\phi(\mathrm{A_{\rho,\Delta-\delta}})$, we
have $\theta=(I+\mathrm{v})^{-1}(\varphi)$. We subsitute this expression
of $\theta$, in equations (\ref{eq: epsilon term}) to (\ref{eq: quad form})
to obtain the following
\begin{prop}
\label{prop:Linearised}Let $H_{\epsilon}\in\mathcal{H}_{\rho,\Delta}$
given by (\ref{eq:ham}) and assume that (\ref{eq:non degen neighboor})
is true. 

Let $u$, $\mathrm{v}$, $\alpha$ be the solution of the system (\ref{eq:system of fundamental equations})
on $\mathrm{A_{\Delta-\delta}}$, for any $0<\delta<\Delta$. 

We assume furthermore that :
\[
\begin{cases}
\mid\mathrm{dv}\mid_{\Delta-\delta}<1\\
\Delta R=\frac{\partial\mathrm{u}}{\partial\theta}+\alpha+^{t}d\mathrm{v}.R & =O(\epsilon)
\end{cases}
\]
 Then $\phi$ is a symplectic diffeomorphism onto its image and $\forall(R,\varphi)\in\phi(\mathrm{A_{\rho,\Delta-\delta}})$
:

\begin{eqnarray*}
H_{\epsilon}\circ\phi^{-1}(R,\varphi) & = & \bar{m}+{}^{t}\omega.R\,+\frac{1}{2}\,^{t}R.\bar{S}(\varphi).R+\bar{\epsilon}\,\bar{h}(R,\varphi)+\bar{g}(R,\varphi)\\
 & = & \bar{m}+{}^{t}\omega.R\,+O(\epsilon^{2})+O(R^{2)}
\end{eqnarray*}

With :
\begin{equation}
\begin{cases}
\bar{m}=m+^{t}\omega.\alpha+\epsilon\intop_{\mathrm{T^{d}}}a(\theta)\mathrm{d}\theta\\
\bar{S}(\varphi)=\epsilon.c(\theta)+2\sum_{k=1}^{d}Q_{k}(\theta).(\Delta\tilde{R})_{k}\\
\bar{\epsilon}\,\bar{h}(R,\varphi)=\epsilon\,\intop_{0}^{\,\,1}t\,\,\frac{\partial h}{\partial R}(R_{t}\,,\theta)\,\mathrm{dt}\,\Delta R+\frac{1}{2}\,^{t}\Delta R\left(\intop_{0}^{\,\,1}\,t^{2}\,\frac{\partial^{2}g}{\partial R^{2}}(R_{t}\,,\theta)\,\mathrm{dt}\right)\Delta R+\frac{1}{2}\,^{t}\Delta R.S(\theta).\Delta R=O(\epsilon^{2})\\
\bar{g}(R,\varphi)=g(R,\theta)+\epsilon\sum_{|k|\geq3}\,h_{k}(\theta)\,R^{k}+{}^{t}\Delta R\,.\frac{\partial g}{\partial R}(R,\theta)-^{t}R\left(\sum_{k=1}^{d}Q_{k}(\theta).(\Delta\tilde{R})_{k}\right)R=O(R^{3})
\end{cases}\label{eq:formula scheme}
\end{equation}

Where $\theta=(I+\mathrm{v})^{-1}(\varphi)$, and $\Delta\tilde{R}$,
$Q_{k}$, $R_{t}$ as defined above.
\end{prop}

\section{KAM quadratic scheme}

We are now in a position to precise the quadratic scheme that will
be crucial for the iterative construction of the symplectic diffeomorphism
in theorem \ref{thm:Kolmogorov-Arnold-Moser}.

Recall that $\omega\in{\rm DC}(c,\tau)$ and that $\hat{S}$ is a
symmetric $d\times d$ real invertible matrix. Let $\tilde{S}$ be
its inverse and $\beta>0$ be such that : 
\[
\forall M\in\mathcal{M}_{n}(\mathbb{C}),\,\,\left|M-\hat{S}\right|<\beta\,\Rightarrow M\,\text{ is invertible and }\left|M^{-1}-\tilde{S}\right|<1
\]

.
\begin{prop}
\label{prop:KAM-Scheme}\textbf{\textup{Quadratic Scheme}}

Let $H\in\mathcal{H}_{\rho,\Delta}$ with the same form as in (\ref{eq:ham})
i.e :

\[
H(r\mathrm{,}\theta)=m+{}^{t}\omega.r\,+\frac{1}{2}\,^{t}r.S(\theta).r+\epsilon\,h(\theta,r)+g(\theta,r)
\]

Where $m$ is a constant, $\omega\in DC(c,\tau)$ , $S(\theta)$ a
symmetric matrix with coefficients in $\mathcal{H_{\mathrm{\rho,\Delta}}}$,
$h,\,g\,\in\mathcal{H_{\mathrm{\rho,\Delta}}}$ with $g(\theta,r)=O\left(r^{3}\right)$
and $|h|_{\rho,\Delta}\leq1$~.

We asssume that $\exists\,\eta_{0}\,,\,\gamma_{0},\,\rho_{0},\,\Delta_{0},\,\epsilon_{0}>0$
~ are given constants such that :
\[
(\mathcal{H}_{1})\quad\begin{cases}
\gamma= & |m|+|S|_{\Delta}+|g|_{\rho,\Delta}+|\hat{S}|+|\tilde{S}|\leq2\gamma_{0}\\
\eta= & \left|S-\hat{S}\right|_{\Delta}\leq\eta_{0}<\beta\\
\rho_{0}\geq & \rho\geq\rho_{0}/2\\
\Delta_{0}\geq & \Delta>\Delta_{0}/2
\end{cases}
\]

Then $\exists\,C(c,\,\tau,\,d,\,\gamma_{0},\,\eta_{0},\,\rho_{0},\,\Delta_{0})$
such that :

\[
(\mathcal{H}_{2})\quad\begin{cases}
\delta\in(0,\min(\rho_{0}/3,\Delta_{0}/3,1))\\
\epsilon<C^{-1}\delta^{2(d+\tau+2)}
\end{cases}
\]
~then there exists a symplectic diffeomorphism $\tilde{\phi}:\mathrm{A_{\bar{\rho},\bar{\Delta}}\rightarrow\mathrm{U\subset A_{\rho,\Delta}}}$
~where $\bar{\rho}=\rho-\delta$ and $\bar{\Delta}=\Delta-\delta$
,~$A_{\rho-5\delta/4,\Delta-5\delta/4}\subseteq\mathrm{U}$ such
that :

\[
H\circ\tilde{\phi}(R\mathrm{,}\varphi)=H(R,\varphi)=\bar{m}+{}^{t}\omega.R\,+\frac{1}{2}\,^{t}R.\bar{S}(\varphi).R+\bar{\epsilon}\,\bar{h}(R,\varphi)+\bar{g}(R,\varphi)
\]

Where $\bar{m}$ is a constant, $\bar{S}(\varphi)$ a symmetric matrix,
$\bar{h},\,\bar{g}\,\in\mathcal{H_{\mathrm{\rho-\delta,\Delta-\delta}}}$
with $\bar{g}(R,\varphi)=O\left(R^{3}\right)$ and $|\bar{h}|_{\rho-\delta,\Delta-\delta}\leq1$~.

Moreover if we let:
\[
\hat{\epsilon}:=|\phi-Id|_{\bar{\rho},\bar{\Delta}}\geq|\phi^{-1}-Id|_{\bar{\rho}-\delta/4,\bar{\Delta}-\delta/4}
\]

then 
\begin{equation}
\hat{\epsilon}\leq C\frac{\epsilon}{\delta^{\nu(\tau,d)}}\label{eq:ineq epsil}
\end{equation}

\end{prop}
and 
\begin{equation}
\begin{cases}
\bar{\gamma} & \leq\gamma+\hat{\epsilon}\\
\bar{\eta} & \leq\eta+\hat{\epsilon}\\
\bar{\epsilon} & \leq C\frac{\epsilon^{2}}{\delta^{2\nu(\tau,d)}}
\end{cases}\label{eq:ineq eps gamm eta}
\end{equation}
for $\nu(\tau,d)=2(d+\tau+2)$.

Where $\bar{\gamma}=|\bar{m}|+|\bar{S}|_{\bar{\Delta}}+|g|_{\bar{\rho},\bar{\Delta}}+|\hat{S}|+|\tilde{S}|$
and $\bar{\eta}=\left|\bar{S}-\hat{S}\right|_{\bar{\Delta}}$
\begin{proof}
Let $a(\theta)=h(0,\theta)$, $b(\theta)=\frac{\partial h}{\partial r}(0,\theta)$
and $c(\theta)=\frac{\partial^{2}h}{\partial r^{2}}(0,\theta)$. By
Cauchy estimates we have that :
\begin{eqnarray*}
|a|_{\Delta} & \leq & 1\\
|b|_{\Delta} & \leq & \frac{2}{\rho_{0}}\\
|c|_{\Delta} & \leq & \frac{8}{(\rho_{0})^{2}}
\end{eqnarray*}

For $\delta^{'}>0$, by the lemma (\ref{lem:Fundamental-lemma}),
let $u\in\mathcal{H}_{\Delta-\delta'}$ be the unique solution of
:

\[
\mathcal{L_{\omega}}u=-\epsilon a(\theta)+\epsilon\intop_{\mathrm{T^{d}}}a(\theta)\mathrm{d}\theta
\]

By the lemma (\ref{lem:Fundamental-lemma}), we have that :
\[
\mid u\,\mid_{\Delta-\delta^{'}}\leq\frac{C(\tau,\,d)}{\delta^{'(d+\tau)}}\,2\epsilon\mid a\,\mid_{\Delta}\leq2\epsilon\frac{C(\tau,\,d)}{\delta^{'(d+\tau)}}
\]

Also by a Cauchy estimate we obtain that : 
\begin{equation}
\left|\frac{\partial u}{\partial\theta}\right|_{\Delta-1.5\delta'}\leq2\frac{\mid u\,\mid_{\Delta-\delta^{'}}}{\delta^{'}}\leq4\epsilon\frac{C(\tau,\,d)}{\delta^{'(d+\tau+1)}}\label{eq:ineq du dthet}
\end{equation}

Also since $\left|S-\hat{S}\right|_{\Delta}\leq\eta_{0}<\beta$ ,
we have that $\intop_{\mathrm{T^{d}}}S(\theta)\mathrm{d}\theta\in B(\hat{S},\,\beta)$~therefore
$\intop_{\mathrm{T^{d}}}S(\theta)\mathrm{d}\theta$ is invertible
moreover we get that $\left|\left(\intop_{\mathrm{T^{d}}}S(\theta)\mathrm{d}\theta\right)^{-1}\right|\leq|\tilde{S}|+1\leq2\gamma_{0}+1$.\\
So now we let :
\[
\alpha=-\left(\intop_{\mathrm{T^{d}}}S(\theta)\mathrm{d}\theta\right)^{-1}\intop_{\mathrm{T^{d}}}\left(S(\theta)\frac{\partial u}{\partial\theta}+\epsilon\,b(\theta)\right)\mathrm{d}\theta
\]

So we have :
\begin{equation}
\left|\alpha\right|\leq(1+2\gamma_{0})\epsilon\left(8\gamma_{0}\frac{C(\tau,\,d)}{\delta^{'(d+\tau+1)}}+\frac{2}{\rho_{0}}\right)\label{eq:alpha}
\end{equation}

And by the lemma (\ref{lem:Fundamental-lemma}), let$\mathrm{v=^{t}(v_{1},..,v_{d})}$
be the unique element of $\left(\mathcal{H}_{\Delta-2\delta'}\right)^{d}$
, such that :

\[
\mathcal{L_{\omega}}\mathrm{v}=-S(\theta).\left(\frac{\partial u}{\partial\theta}+\alpha\right)-\epsilon\,b(\theta)
\]

Where $\mathcal{L_{\omega}}\mathrm{v}=^{t}(\mathcal{L_{\omega}}\mathrm{v}_{1},...,\mathcal{L_{\omega}}\mathrm{v}_{d})$
.

Hence we deduce easily that we have a constant $C_{\mathrm{v}}$ such
that :
\begin{eqnarray}
|\mathrm{v}|_{\Delta-2\delta'} & \leq & \epsilon\frac{C_{\mathrm{v}}}{\delta'^{(2d+2\tau+1)}}\label{eq:ineq v}\\
|\mathrm{dv}|_{\Delta-3\delta'} & \leq & \epsilon\frac{C_{\mathrm{v}}}{\delta'^{2(d+\tau+1)}}
\end{eqnarray}

Notice that if $\delta'^{2(d+\tau+1)}>2C_{\mathrm{v}}.\epsilon$ (i.e
$|\mathrm{dv}|_{\Delta-3\delta'}<\frac{1}{2}$) then $Id+\mathrm{v}$
is injective on $\mathrm{A_{\Delta-3\delta'}}$ ~and hence a diffeomorphism
on its image. Moreover since under this condition $|\mathrm{v}|_{\Delta-2\delta'}<\delta'$,
the image of $Id+\mathrm{v}$ contains $\mathrm{A_{\Delta-4\delta'}}$.
So on $A_{\rho,\Delta-3\delta'}$, we have the symplectic diffeomorphism
onto its image, given by (\ref{eq:symplectic diffeomorphism}). Also
we get that :

\begin{eqnarray*}
|R-r|_{\rho,\Delta-3\delta'} & \leq & \left|\left(Id+^{t}\mathrm{dv}\right)^{-1}-Id\right|_{\Delta-3\delta'}\rho+\left|\left(Id+^{t}\mathrm{dv}\right)^{-1}\right|_{\Delta-3\delta'}\left(|\alpha|+\left|\frac{\partial u}{\partial\theta}\right|_{\Delta-1.5\delta'}\right)
\end{eqnarray*}

So we have a constant $C_{\mathrm{R}}(\tau,d,\gamma_{0},\rho_{0},\Delta_{0})$
such that :
\begin{equation}
|R-r|_{\rho,\Delta-3\delta'}\leq\epsilon\frac{C_{\mathrm{R}}}{\delta'^{2(d+\tau+1)}}\label{eq:ineq delta R}
\end{equation}

And by (\ref{eq:ineq v}) and (\ref{eq:ineq delta R}) :

\begin{equation}
|\phi-Id|_{\rho,\Delta-3\delta'}\leq\epsilon\frac{\max(C_{\mathrm{R}},\,C_{\mathrm{v}}.\Delta_{0}/2)}{\delta'^{2(d+\tau+1)}}\label{eq: phi major}
\end{equation}

On $B=\phi(A_{\rho,\Delta-3\delta'})$, with $w=(I+\mathrm{v})^{-1}$~we
have : 
\[
\begin{cases}
\theta= & \varphi-\mathrm{v}(w(\varphi))\\
r\,= & \frac{\partial\mathrm{u}}{\partial\theta}+\alpha+R+^{t}d\mathrm{v}.R
\end{cases}
\]

For~ $\phi^{-1}=\tilde{\phi}:\,B\rightarrow A_{\rho,\Delta-3\delta'}$,
we clearly have: 
\begin{equation}
|\phi-Id|_{\rho,\Delta-3\delta'}=|\phi^{-1}-Id|_{B}\leq\epsilon\frac{C_{\phi}}{\delta'^{2(d+\tau+1)}}\label{eq: phi inv maj}
\end{equation}

If $\epsilon.\max(C_{\mathrm{R}},\,C_{\mathrm{v}}.\Delta_{0}/2)<\delta'^{2(d+\tau)+3}$,
then $\phi$ defined by (\ref{eq:symplectic diffeomorphism}) is a
diffeomorphism on $A_{\rho,\Delta-3\delta'}$ and $\phi(A_{\rho,\Delta-3\delta'})\supseteq A_{\rho-\delta',\Delta-4\delta'}$.
~Moreover, ~$\phi^{-1}(A_{\rho-4\delta^{'},\Delta-4\delta'})\supseteq A_{\rho-5\delta',\Delta-5\delta'}$.

By setting $\delta=4\,\delta^{'}$ and $C_{\delta}=4^{2(d+\tau)+3}\,\max(C_{\Delta R},C_{\mathrm{v}}.\Delta_{0}/2,\,2C_{\mathrm{v}})$,
we have just shown that if : 
\begin{eqnarray}
\delta & < & \min\left(1,\frac{\rho_{0}}{3},\frac{\Delta_{0}}{3}\right)\nonumber \\
\epsilon & < & C_{\delta}^{-1}\delta^{2(d+\tau+2)}\label{eq:epsilon h1}
\end{eqnarray}
Then $\phi$ is a diffeomorphism from $A_{\rho,\Delta-3\delta/4}$
onto a domain that contains $A_{\rho-\delta/4,\Delta-3\delta/4}$
and that $\phi^{-1}(A_{\rho-\delta,\Delta-\delta})\supseteq A_{\rho-5\delta/4,\Delta-5\delta/4}$.

And we have more generally that $\forall\tilde{\rho}\leq\rho,\,\tilde{\Delta}\leq\Delta$
\begin{equation}
A_{\tilde{\rho}-5\delta/4,\tilde{\Delta}-5\delta/4}\subseteq\phi^{-1}(A_{\tilde{\rho}-\delta,\tilde{\Delta}-\delta})\subseteq A_{\tilde{\rho}-3\delta/4,\tilde{\Delta}-3\delta/4}\label{eq:domain analyticity}
\end{equation}

Let $\tilde{\phi}$ be the restriction of $\phi^{-1}$ to $A_{\rho-\delta,\Delta-\delta}$,
then by (\ref{eq: phi major}) and (\ref{eq: phi inv maj}), we have
a constant $C_{\mathrm{\phi}}(\tau,d,\gamma_{0},\rho_{0},\Delta_{0})$
such that :
\[
\hat{\epsilon}=|\tilde{\phi}-Id|_{\rho-\delta,\Delta-\delta}\leq C_{\phi}\frac{\epsilon}{\delta^{2\tau+2d+2}}
\]

and : 
\[
|\tilde{\phi}^{-1}-Id|_{\rho-5\delta/4,\Delta-5\delta/4}\leq\epsilon\frac{C_{\phi}}{\delta{}^{2(d+\tau+1)}}
\]

By proposition \ref{prop:Linearised}, we have then :
\[
H\circ\tilde{\phi}(R\mathrm{,}\varphi)=H(R,\varphi)=\bar{m}+{}^{t}\omega.R\,+\frac{1}{2}\,^{t}R.\bar{S}(\varphi).R+\bar{\epsilon}\,\bar{h}(R,\varphi)+\bar{g}(R,\varphi)
\]

And by (\ref{eq:formula scheme}) , we have: 
\begin{align*}
\bar{m}= & m+^{t}\omega.\alpha+\epsilon\intop_{\mathrm{T^{d}}}a(\theta)\mathrm{d}\theta\\
\bar{S}(\varphi)= & S(\theta)+\epsilon.c(\theta)+2\sum_{k=1}^{d}Q_{k}(\theta).(\Delta\tilde{R})_{k}\\
\bar{g}(R,\varphi)= & g(R,\theta)+\epsilon\sum_{|k|\geq3}\,h_{k}(\theta)\,R^{k}+{}^{t}\Delta R\,.\frac{\partial g}{\partial R}(R,\varphi)-^{t}R\left(\sum_{k=1}^{d}Q_{k}(\theta).(\Delta\tilde{R})_{k}\right)R\\
\bar{\epsilon}\bar{h}(R,\varphi)= & \epsilon\,\intop_{0}^{\,\,1}t\,\,\frac{\partial h}{\partial R}(R_{t}\,,\theta)\,\mathrm{dt}\,\Delta R+\frac{1}{2}\,^{t}\Delta R\left(\intop_{0}^{\,\,1}\,t^{2}\,\frac{\partial^{2}g}{\partial R^{2}}(R_{t}\,,\theta)\,\mathrm{dt}\right)\Delta R+\frac{1}{2}\,^{t}\Delta R.S(\theta).\Delta R
\end{align*}

Where 
\[
\theta=\varphi-\mathrm{v}(w(\varphi))
\]

For $\bar{m}$, by (\ref{eq:alpha}), $\exists C_{m}$ a constant
(that is a function of the fixed parameters above) such that :

\[
|\bar{m}|\leq|m|+\epsilon\frac{C_{m}}{\delta{}^{2(d+\tau+2)}}
\]

Also $Q_{k}(\theta)=\frac{1}{6}\frac{\partial g}{\partial R_{k}\partial R_{i}\partial R_{j}}(0,\,\theta)$
, $\theta=\varphi-\mathrm{v}(w(\varphi))\in A_{\Delta-3\delta/4}$
and $g$ is analytical on $A_{\rho,\Delta}$ , so by a Cauchy estimate
: 
\[
|Q_{k}(\theta)|_{\Delta-3\delta/4}\leq8\frac{\gamma_{0}}{\delta^{3}}
\]

And by (\ref{eq:ineq du dthet}) and (\ref{eq:alpha}) we have $C_{S}$
a constant such that :

\[
\left|\epsilon.c(\theta)+2\sum_{k=1}^{d}Q_{k}(\theta).(\Delta\tilde{R})_{k}\right|_{\rho-\delta,\Delta-\delta}\leq\epsilon\frac{C_{S}}{\delta^{d+\tau+4}}
\]

So :
\[
\bar{\eta}\leq\eta+\epsilon\frac{C_{S}}{\delta^{d+\tau+4}}\leq\eta+\epsilon\frac{C_{S}}{\delta^{2(d+\tau+2)}}
\]

Similarly by a Cauchy estimate since $(R,\theta)\in A_{\rho-\delta,\Delta-3\delta/4}$,
we have that :

\[
\left|\frac{\partial g}{\partial R}(R,\theta(\phi))\right|_{\rho-\delta,\Delta-\delta}\leq2\gamma_{0}/\delta
\]

So we have a constant $C_{g}$ such that: 
\[
|\bar{g}|_{\rho-\delta,\Delta-\delta}\leq|g|_{\rho,\Delta}+C_{g}\epsilon\frac{1}{\delta{}^{2(d+\tau+2)}}
\]

Finally we notice that $\forall t\in[0,1]\,\,\,(R_{t},\theta(\varphi))\in A_{\rho-\delta,\Delta-3\delta/4}$
~because of (\ref{eq:domain analyticity}) so by a Cauchy estimate
and previous estimates we get a constant $C_{\epsilon}$ such that
:

\[
\bar{\epsilon}|\bar{h}|_{\rho-\delta,\Delta-\delta}\leq C_{\epsilon}\epsilon^{2}\frac{1}{\delta^{4(d+\tau+2)}}
\]

With $\bar{\epsilon}=C_{\epsilon}\epsilon^{2}\frac{1}{\delta^{4(d+\tau+2)}}$
, we get $\bar{h}|_{\rho-\delta,\Delta-\delta}\leq1$ .

Now if set $C=C_{\epsilon}+C_{\phi}+C_{\delta}+C_{m}+C_{S}+C_{g}$
, then if ($\mathcal{H}_{2}$ ) holds true , so does (\ref{eq:epsilon h1})
and therefore we have the symplectic diffeomorphism with the required
estimates.
\end{proof}

\section{Iteration of the quadratic scheme }

We need first the following bounding lemma :
\begin{lem}
\label{lem:Majoration-lemma}Bounding lemma
\[
\forall C_{1},C_{2},C_{3},C_{4},c_{1},c_{2}>0,\,\,\,\,\exists\kappa>0\,\,\,\,\,(|x|<\kappa)\Rightarrow\begin{cases}
\sup_{n\geq0}\left(|x|^{2^{n}}(C_{1})^{n}(C_{2})^{n^{2}}\right) & <c_{1}\\
\sum_{n=0}^{\infty}|x|^{2^{n}}(C_{3})^{n}(C_{4})^{n^{2}} & <c_{2}
\end{cases}
\]
\end{lem}
\begin{proof}
Let $C_{1},\,c_{1}>0$, then write for any $0<x<1$ and for any $n\geq0$:

\[
x^{2^{n}}(C_{1})^{n}<\exp\left(n.(\ln(x)\ln(2)+\ln(C_{1}))+n^{2}\left(\frac{1}{2}\ln(x)\ln(2)+\ln(C_{2})\right)+\ln(x)\ln(2)\right)
\]

If $\left(x<\exp\left(\frac{1}{\ln(2)}\min(-2\ln(C_{2}),-\ln(C_{1}),\ln(c_{1}))\right)=\kappa_{1}\right)$
then we have indeed :
\[
\forall n\in\mathbb{N},\,x^{2^{n}}(C_{1})^{n}(C_{2})^{n^{2}}<c_{1}
\]

In particular, if $\left(x<\exp\left(\frac{1}{\ln(2)}\min(-2\ln(2C_{4}),-\ln(2C_{3}),\ln(c_{2}))\right)=\kappa_{2}\right)$
, 
\[
\forall n\in\mathbb{N},\,x^{2^{n}}(C_{3})^{n}(C_{4})^{n^{2}}<c_{2}\frac{1}{2^{n+n^{2}}.3}
\]

So by summing this inequality, we get both inequalities for $x<\kappa=\min(\kappa_{1},\kappa_{2})$.
\end{proof}
We can easily prove a similar inequality with any polynom $P$ :\linebreak{}
\[
(\forall c>0,\,\exists\kappa>0),\,\,(0<x<\kappa)\Rightarrow\,\,\left(\sum_{n=0}^{\infty}x^{2^{n}}\exp(P(n))\right)<c
\]
.

We proceed now to the iterative construction 

We let $\rho_{0}=\rho$ , $\Delta_{0}=\Delta$ and $\forall n\geq1$
:
\[
\begin{cases}
\delta_{n}=\frac{1}{3^{n}}\min(\rho_{0},\Delta_{0},1)\\
\rho_{n}=\rho_{n-1}-\delta_{n}\\
\Delta_{n}=\Delta_{n-1}-\delta_{n}
\end{cases}
\]

Clearly $\forall n\geq0\,\,\Delta_{0}\geq\Delta_{n}\geq\Delta_{0}/2$~and
$\,\rho_{0}\geq\rho_{n}\geq\rho_{0}/2$. Let $\delta=\min(\rho_{0},\Delta_{0},1)$.

Also with $H_{\epsilon}$ as in theorem \ref{thm:Kolmogorov-Arnold-Moser}
i.e $H_{\epsilon}(r,\theta)=f_{0}(r)+\epsilon f_{1}(r,\theta)$, we
make obvious substitutions to put in the following form:
\[
H_{\epsilon}(r,\theta)=m_{0}+{}^{t}\omega.r\,+\frac{1}{2}\,^{t}r.S_{0}.r+\epsilon_{0}\,h_{0}(r,\theta)+g_{0}(R,\theta)
\]

Where $m_{0}=f_{0}(0)$ is a constant, $\omega=\frac{\partial f_{0}}{\partial r}(0)\in DC(c,\tau)$
, $S_{0}=\frac{\partial^{2}f_{0}}{\partial r^{2}}(0)$ an invertible
symmetric matrix ,$h_{0},\,g_{0}\,\in\mathcal{H_{\mathrm{\rho,\Delta}}}$
with $g_{0}(r,\theta)=O\left(r^{3}\right)$ and $|h_{0}|_{\rho,\Delta}\leq1$~.

Also previously, let $\hat{S}=S_{0}$ and let $\tilde{S}$ be its
inverse and $\beta>0$ be such that : 
\[
\forall M\in\mathcal{M}_{n}(\mathbb{C}),\,\,\left|M-\hat{S}\right|\leq\beta\,\Rightarrow M\,\text{ is invertible and }\left|M^{-1}-\tilde{S}\right|\leq1
\]
Finally we let $H_{0}=H_{\epsilon}$ , $\gamma_{0}=|S_{0}|+|g_{0}|_{\rho_{0},\Delta_{0}}+|m_{0}|+|S_{0}^{-1}|$
and $\eta_{0}=\beta$.

Also let $C(c,\,\tau,\,d,\,\gamma_{0},\,\eta_{0},\,\rho_{0},\,\Delta_{0})$
be the constant given by the proposition \ref{prop:KAM-Scheme} and
$\nu=2.(d+\tau+2)$
\begin{prop}
Inductive s\label{prop:Inductive-step}tep

$\exists\kappa>0$, if $\epsilon_{0}\in(0,\kappa)$, we have a sequence
$(\phi_{n})_{n\geq1}$ of real analytic symplectic diffeomorphism
with domain $A_{\rho_{n},\Delta_{n}}$ such that :

\[
\phi_{n}\left(A_{\rho_{n},\Delta_{n}}\right)\subseteq A_{\rho_{n-1},\Delta_{n-1}}
\]

And if $\forall n\geq1$, $H_{n}=H_{n-1}\circ\phi_{n}\in\mathcal{H}_{\rho_{n},\Delta_{n}}$
then :

\[
H_{n}(r,\theta)=m_{n}+{}^{t}\omega.r\,+\frac{1}{2}\,^{t}r.S_{n}.r+\epsilon_{n}\,h_{n}(r,\theta)+g_{n}(R,\theta)
\]

Where $m_{n}$ is a constant $S_{n}$ a symmetric matrix ,$h_{n},\,g_{n}\,\in\mathcal{H_{\mathrm{\rho_{n},\Delta_{n}}}}$
with $g_{n}(r,\theta)=O\left(r^{3}\right)$ and $|h_{n}|_{\rho,\Delta}\leq1$

Such that $\forall n\geq1$ 
\[
\epsilon_{n}<C^{-1}\delta_{n+1}^{\nu}
\]
 
\begin{equation}
\epsilon_{n}\leq C^{n}\epsilon_{0}^{2^{n}}\delta^{-2\nu n}3^{\nu n(n+1)}\label{eq:cond eps}
\end{equation}

\end{prop}
and
\begin{equation}
\hat{\epsilon}_{n}\leq C^{n}\epsilon_{0}^{2^{n-1}}\delta^{\nu}3^{\nu n^{2}}\delta^{-2\nu n}\label{eq:cond eps hat}
\end{equation}
 
\begin{equation}
\begin{cases}
\gamma_{n} & \leq\gamma_{n-1}+\hat{\epsilon}_{n}\leq2\gamma_{0}\\
\eta_{n} & \leq\eta_{n-1}+\hat{\epsilon}_{n}\leq\eta_{0}
\end{cases}\label{eq:cond bounds}
\end{equation}

Where $\gamma_{n}=|m_{n}|+|S_{n}|_{\Delta_{n}}+|g_{n}|_{\rho_{n},\Delta_{n}}+|\hat{S}|+|\tilde{S}|$,
$\eta_{n}=\left|S_{n}-\hat{S}\right|_{\Delta_{n}}$

And 
\[
\hat{\epsilon}_{n}:=|\phi-Id|_{\rho_{n},\Delta_{n}}\geq|\phi^{-1}-Id|_{\rho_{n}-\delta_{n}/4,\Delta_{n}-\delta_{n}/4}
\]

Moreover $\forall\tilde{\rho}\in(\delta_{n}/4,\,\rho_{n}]$,~$\forall\tilde{\Delta}\in(\delta_{n}/4,\,\Delta_{n}]$,
\begin{equation}
\phi_{n}(A_{\tilde{\rho},\tilde{\Delta}})\supseteq A_{\tilde{\rho}-\delta_{n}/4,\tilde{\Delta}-\delta_{n}/4}\label{eq:domain by phi map}
\end{equation}

And :
\[
\begin{cases}
\sum_{n=1}^{\infty} & \hat{\epsilon}_{n}<\infty\\
\lim_{n\rightarrow\infty} & \epsilon_{n}=0
\end{cases}
\]

\begin{proof}
Let $\kappa$ given by the bounding lemma such that

\begin{equation}
(|x|<\kappa)\Rightarrow\begin{cases}
\sup_{n\geq0}\left(|x|^{2^{n}}(C\delta^{-2\nu})^{n}3^{\nu}{}^{(n+1)^{2}}\right) & <C^{-1}\delta^{\nu}\\
\sum_{n=0}^{\infty}|x|^{2^{n}}(C\delta^{-2\nu})^{(n+1)}3^{\nu(n+1)^{2}} & <\min(\eta_{0},\gamma_{0})\delta^{-\nu}
\end{cases}\label{eq: condition H1}
\end{equation}

The first condition implies that if for a given $n\geq0$ , (\ref{eq:cond eps})
is true, then $\epsilon_{n}<C^{-1}\delta_{n+1}^{\nu}$ so condition
$(\mathcal{H}_{2})$ is true.

The second condition impies that if $\forall k\in\left\llbracket 1,n\right\rrbracket $,
(\ref{eq:cond eps hat}) then ${\displaystyle \sum_{k=1}^{n}\hat{\epsilon}_{k}<\min(\eta_{0},\gamma_{0})}$.

We will use these two facts to ensure during the induction that ($\mathcal{H}_{1}$)
and ($\mathcal{H}_{2}$) of proposition \ref{prop:KAM-Scheme} are
true for every $n$.

For $n=1$ we have : 
\[
\epsilon_{0}<C^{-1}\delta_{1}^{\nu}
\]

So for the Hamiltonian $H_{0},$ $(\mathcal{H}_{1})$ is true and
($\mathcal{H}_{2}$) is clearly true as seen before the statement
of the proposotion. By proposition \ref{prop:KAM-Scheme}, we have
a symplectic diffeomorphism $\phi_{1}$ with $\phi_{1}\left(A_{\rho_{1},\Delta_{1}}\right)\subseteq A_{\rho_{0},\Delta_{0}}$,
$H_{1}=H_{0}\circ\phi_{1}\in\mathcal{H}_{\rho_{1},\Delta_{1}}$and
:
\[
\begin{cases}
\hat{\epsilon}_{1}\leq & C\epsilon_{0}/\delta_{1}^{\nu}=3^{\nu}C\epsilon_{0}/\delta^{\nu}\\
\epsilon_{1}\leq & C\epsilon_{0}^{2}/\delta_{1}^{\nu}\\
\gamma_{1}\leq & \gamma_{0}+\hat{\epsilon}_{1}\le2\gamma_{0}\\
\eta_{1}\leq & \hat{\epsilon}_{1}\leq\eta_{0}
\end{cases}
\]

Where the last two inequalities are due to $\hat{\epsilon}_{1}<\min(\gamma_{0}\eta_{0})$.

Moreover the relation (\ref{eq:domain by phi map}) results from (\ref{eq:domain analyticity})
obtained during the proof of the quadratic scheme. This shows that
our proposition is true for $n=1$.

If we have $\phi_{1},...,\phi_{n}$ symplectic diffeomorphisms as
in the proposition, with the corresponding hamiltonians such that
(\ref{eq:cond eps}) to (\ref{eq:domain by phi map}) are true.

The first remark made before the begining of this induction show that
for the Hamiltonian $H_{n},$ $(\mathcal{H}_{2})$ is true and ($\mathcal{H}_{1}$)
is true by the induction hypothesis.By proposition \ref{prop:KAM-Scheme},
we have a symplectic diffeomorphism $\phi_{n+1}$ with $\phi_{n+1}\left(A_{\rho_{n+1},\Delta_{n+1}}\right)\subseteq A_{\rho_{n},\Delta_{n}}$,
$H_{n+1}=H_{n}\circ\phi_{n+1}\in\mathcal{H}_{\rho_{n+1},\Delta_{n+1}}$and
:

\[
\begin{cases}
\hat{\epsilon}_{n+1}\leq & C\epsilon_{n}/\delta_{n+1}^{\nu}\\
\epsilon_{n+1}\leq & C\epsilon_{n}^{2}/\delta_{n+1}^{\nu}\\
\gamma_{n+1}\leq & \gamma_{n}+\hat{\epsilon}_{n+1}\\
\eta_{n+1}\leq & \eta_{n}+\hat{\epsilon}_{n+1}
\end{cases}
\]

From the first two inequalities and the induction hypothesis, we get
directly for $n+1$, the relations (\ref{eq:cond eps}) and (\ref{eq:cond eps hat}).
From the last two inequalities and the induction hypothesis we get
:
\[
\begin{cases}
\gamma_{n+1}\leq & \gamma_{n}+\hat{\epsilon}_{n+1}\leq\gamma_{0}+\sum_{k=1}^{n+1}\hat{\epsilon}_{k}\\
\eta_{n+1}\leq & \eta_{n}+\hat{\epsilon}_{n+1}\leq\sum_{k=1}^{n+1}\hat{\epsilon}_{k}
\end{cases}
\]

Because of the preliminary remark in relation to (\ref{eq: condition H1}),
we obtain (\ref{eq:cond bounds}). Again the relation (\ref{eq:domain by phi map})
results from (\ref{eq:domain analyticity}) obtained during the proof
of the quadratic scheme. And this finishes the induction.

Also (\ref{eq: condition H1}) leads easily to both $\sum_{n=1}^{\infty}\hat{\epsilon}_{n}<\infty$
and $\lim_{n\rightarrow\infty}\epsilon_{n}=0$.
\end{proof}

\section{Convergence and conclusion}

Let $\epsilon_{0}\in(0,\kappa)$ with $\kappa$ such that $\phi_{n}$
is the sequence of the symplectic diffeomorphism given by proposition
\ref{prop:Inductive-step} , $H_{n}$ the corresponding sequence of
Hamiltonians and relations (\ref{eq:cond eps}) to (\ref{eq:domain by phi map})
hold.
\begin{proof}
of Theorem \ref{thm:Kolmogorov-Arnold-Moser}

We have proved that $S_{n}$ is a bounded sequence in $\mathcal{H}_{\Delta_{0}/2}$,
and $g_{n}$is a bounded sequence in $\mathcal{H}_{\rho_{0}/2,\Delta_{0}/2}$,
so by Ascoli theorem, we can extract from these sequences, convergent
subsequences.

$m_{n}$ is bounded so has a convergent subsequence too and $\epsilon_{n}h_{n}$
has clearly a zero limit.

So we have $(n_{k})_{k\geq0}$ such that : 
\[
\lim_{k\rightarrow\infty}\,\left|H_{n_{k}}\right|_{\rho_{0}/2,\Delta_{0}/2}=^{t}\omega.r+O(r^{2})
\]
.

We consider now the sequence of symplectic maps $\bar{\psi}_{n}=\phi_{1}\circ..\circ\phi_{n}$
for $n\geq1$ and $\bar{\psi}_{0}=Id$

Because of (\ref{eq:domain by phi map}) we have by an easy induction
and by letting $\sigma_{n}={\displaystyle \sum_{1}^{n}\delta_{k}/4}$,
that:
\[
\bar{\psi}_{n}(A_{\rho_{n},\Delta_{n}})\supseteq A_{\rho_{n}-\sigma_{n},\Delta_{n}-\sigma_{n}}=A_{\rho_{0}-5\sigma_{n},\Delta_{0}-5\sigma_{n}}\supseteq A_{3\rho/8,\,3\Delta/8}
\]

We consider then $\psi_{n}=\bar{\psi}_{n}^{-1}|_{V}$ where $V=A_{3\rho/8,\,3\Delta/8}$.

We have $\psi_{n}=\phi_{n}^{-1}\circ..\circ\phi_{1}^{-1}$ on $V$,
so since $\hat{\epsilon}_{n}<C\epsilon_{n}/\delta_{n}^{v}$ , we get
for all $n\ge1$:

\[
|\psi_{n}-\psi_{n-1}|_{3\rho/8,3\Delta_{0}/8}<C\epsilon_{n}/\delta_{n}^{\nu}
\]

Since :
\[
\sum_{k=1}^{\infty}\hat{\epsilon}_{k}<\infty
\]

Then :

\[
\sum_{n=1}^{\infty}|\psi_{n}-\psi_{n-1}|_{3\rho/8,\,3\Delta/8}<\infty
\]

So the sequence $\psi_{n}$ is convergent to a limit $\psi_{\epsilon}$
symplectic from $A_{3\rho/8,3\Delta/8}$ and moreover since $H_{\epsilon}\circ\psi_{n_{k}}=H_{n_{k}}$,
we get that : $H_{\epsilon}\circ\psi_{\epsilon}\,={}^{t}\omega.r+O(r^{2})$.

Now we have : 
\[
|Id-\psi_{\epsilon}|_{3\rho/8,\,3\Delta/8}\leq|\sum_{n=1}^{\infty}|\psi_{n}-\psi_{n-1}|_{3\rho/8,\,3\Delta/8}\leq\sum_{n=0}^{\infty}C\epsilon_{n+1}/\delta_{n+1}^{\nu}\leq C\epsilon_{0}
\]
 So if $\epsilon_{0}$ is sufficiently small, by the Cauchy inequality
$|I-\mathrm{d}\psi_{\epsilon}|_{\rho/3,\Delta/3}<1/2$ therefore $\psi_{\epsilon}$
is injective and is a symplectic diffeomorphism from $A_{\rho/3,\Delta/3}$
onto its image.
\end{proof}

\section{References}
\bibliographystyle{amsalpha}
\bibliography{biblio.bib}

\end{document}